\documentclass[preprint,10pt]{elsarticle}
\usepackage{amssymb}
\usepackage{amsfonts}
\usepackage{amsmath}
\usepackage{amsthm}
\usepackage{epstopdf}
\usepackage{graphicx}
\usepackage{amsmath}
\usepackage{color}
\usepackage{ulem}
\usepackage{multirow}
\usepackage{tikz}
\usepackage{algorithm}
\usepackage{algpseudocode}
\usepackage{algorithmicx}
\usepackage{booktabs}
\usepackage{mathtools}
\usepackage[scr=boondoxupr]{mathalpha}

\usetikzlibrary{decorations.markings}
\usepackage[top=3cm,bottom=4cm,right=3cm,left=3cm]{geometry}
\usepackage{diagbox}

\journal{Currently being prepared
}

\begin{document}
\begin{frontmatter}
\title{(Currently being prepared)A new class of regularized preconditioners for double saddle-point problems}
\author[1]{A. Badahmane}
\address[1]{The UM6P Vanguard Center, Mohammed VI Polytechnic University,, Morocco, email: badahmane.achraf@gmail.com}
\newcommand{\Leg}{{\mathcal L(E,G)}}
\newcommand{\Lef}{{\mathcal L(E,F)}}
\newcommand{\Lfg}{{\mathcal L(F,G)}}
\newcommand{\Le}{{\mathcal L(E)}}
\newcommand{\I}{{\mathcal{I}}}
\newcommand{\ia}{{\mathfrak I}}
\newcommand{\vi}{\emptyset}
\newcommand{\di}{\displaystyle}
\newcommand{\Om}{\Omega}
\newcommand{\na}{\nabla}
\newcommand{\wi}{\widetilde}
\newcommand{\al}{\alpha}
\newcommand{\be}{\beta}
\newcommand{\ga}{\gamma}
\newcommand{\Ga}{\Gamma}
\newcommand{\e}{\epsilon}
\newcommand{\la}{\lambda}
\newcommand{\De}{\Delta}
\newcommand{\de}{\delta}
\newcommand{\entraine}{\Longrightarrow}
\newcommand{\inj}{\hookrightarrow}
\newcommand{\recip}{\Longleftarrow}
\newcommand{\ssi}{\Longleftrightarrow}
\newcommand{\K}{\mathbbm{K}}
\newcommand{\A}{\mathcal{A}}
\newcommand{\R}{\mathbb{R}}
\newcommand{\C}{\mathbb{C}}
\newcommand{\N}{\mathbb{N}}
\newcommand{\Q}{\mathbb{Q}}
\newcommand{\1}{\mathbb{1}}
\newcommand{\0}{\mathbb{0}}
\newcommand{\Z}{\mathbbm{Z}}
\newcommand{\E}{\mathbbm{E}}
\newcommand{\F}{\mathbbm{F}}
\newcommand{\B}{\mathbbm{B}}
\newcommand{\M}{\mathcal{M}_{n}(\K)}

\font\bb=msbm10

\def\ent{{{\rm Z}\mkern-5.5mu{\rm Z}}}
\newtheorem{exo}{Exercice}
\newtheorem{idea}{Idea}

\newtheorem{pre}{Preuve}
\newtheorem{pro}{Propriété}
\newtheorem{exe}{Example}
\newtheorem{theorem}{Theorem}[section]
\newtheorem{proposition}{Proposition}
\newtheorem{definition}{Definition}[section]
\newtheorem{remark}{Remark}[section]
\newtheorem{lem}{Lemma}[section]
\begin{abstract}
The block structure of double saddle-point problems has prompted extensive research into efficient preconditioners. This paper introduces a novel class of three-by-three block preconditioners tailored for such systems from the time-dependent Maxwell
equations  or liquid crystal director modeling. The main motivation of this work is to highlight the limitations of recent preconditioners under high Reynolds numbers, as the original studies did not explore this scenario, and to demonstrate that our preconditioner outperforms the existing ones in such regimes. We thoroughly analyze the convergence and spectral properties of the proposed preocnditioner. We illustrate the efficiency of the proposed preconditioners, and
verify the theoretical bounds.
\end{abstract}
\end{frontmatter}
\textbf{Notations.}
 The notation $\left(x;y;z\right)$ represents a column vector of dimension $N=n+m +l$. Throughout the paper, the identity matrix  denoted by $I$, will be used as needed with its size clear from the context.
\section{Introduction}\label{sec1}
In this paper,we investigate the performance of  new block preconditioner which can be used in the GMRES method \cite{gmres} in solving a class of $3\times 3$ block saddle-point systems of the form:
\begin{eqnarray}
\label{saddle}
\mathcal{A}_{-}\mathrm{w}=\left(\begin{array}{ccc}
A & B^{T}  &0  \\  -B & 0  & -C^{T} \\ 0 & C  & 0
\end{array} \right)\left(\begin{array}{c}
 \mathrm{x}\\ \mathrm{y}\\
 \mathrm{z}
\end{array} \right)=\underbrace{\left(\begin{array}{c}
\mathrm{f}\\ -\mathrm{g}\\
\mathrm{h}
\end{array} \right)}_{\mathrm{b}_{-}},\quad \mathcal{A}_{+}\mathrm{w}=\left(\begin{array}{ccc}
A & B^{T}  &0  \\  B & 0  & C^{T} \\ 0 & C  & 0
\end{array} \right)\left(\begin{array}{c}
 \mathrm{x}\\ \mathrm{y}\\
 \mathrm{z}
\end{array} \right)=\underbrace{\left(\begin{array}{c}
\mathrm{f}\\ \mathrm{g}\\
\mathrm{h}
\end{array} \right)}_{\mathrm{b}_{+}}
\end{eqnarray} where $A\in\mathbb{R}^{n\times n}$ denotes the symmetric positive definite (SPD) matrix, $B\in\mathbb{R}^{m\times n}$ and $C\in\mathbb{R}^{l\times m}$ represent  full row rank rectangular matrices, the vectors $\mathrm{f}\in\mathbb{R}^{n}$, $\mathrm{g}\in\mathbb{R}^{m}$ and $\mathrm{h}\in\mathbb{R}^{l}$, are vectors. Under these conditiones ensure that the coefficient matrix $\mathcal{A}$ of the linear equation (\ref{saddle}) is nonsingular. Linear systems such as  (\ref{saddle}) play a crucial role, including, the least squares problems \cite{Yuan1996},
finite element method for solving time-dependent Maxwell equations \cite{Chen2000}, the Ossen iteration method for the stationary incompressible magnetohydrodynamics systems \cite{YangZhang2020,HuXu2015}. Similar block structures also appear in liquid crystal director modeling or  the coupled Stokes-Darcy problem \cite{Chen2000,Cai2009,Chen2023}. Owing to  the $3 \times 3$ block structure of the saddle-point matrix in (\ref{saddle}), it is hard to present proper preconditioning technologies for solving the $3 \times 3$ block saddle-point problems. Of course, it is easy to see that the saddle-point matrix $\mathcal{A}$ can be viewed as a $2\times 2$ block matrix. In this case, generally speaking, there have been a lot of effective iterative methods for solving linear systems of the $2\times 2$  saddle-point form, see \cite{Benzi}. For example,  successive overrelaxation (SOR)-like methods \cite{Parlett, Wang2, Golub2,Guo}, Uzawa-type methods \cite{Wang1,Parlett, Wang2,Bramble,Zhang}, Hermitian and skew-Hermitian splitting  (HSS) iteration method, which was initially  introduced by Bai, Golub, and Ng in \cite{HSS}. Additionally, the  preconditioned HSS  (PHSS) iteration method was presented in \cite{PHSS}. Arguably the most prominent Krylov subspace methods for solving (\ref{saddle}) are preconditioned variants of MINRES \cite{Minres}  and GMRES \cite{gmres}. In contrast to GMRES, the previously-discovered MINRES algorithm can explicitly exploit the symmetry of $\mathcal{A}$.
However, due to its complex structure, directly solving the standard $2\times 2$ block saddle-point problem often requires more time and iterations. To improve the efficiency of iterative methods, 
new efficient preconditioners for the $3\times 3$ block saddle-point problems are employed, all combining in different ways the Schur complement matrices $S = BA^{-1}B^T$, and $X = CS^{-1}C^T$.  The  block diagonal (BD)  preconditioner \cite{Huang2019}:
\begin{equation}
P_{BD} = \begin{pmatrix}
A & 0 & 0 \\
0 & S & 0 \\
0 & 0 & C S^{-1}C^T
\end{pmatrix},
\end{equation}
where $S = B A^{-1}B^T$. Although the BD preconditioner matrix has a well-distributed set of eigenvalues, constructing the BD preconditioner is very time-consuming. Xie and Li \cite{Li} presented three preconditioners for solving (\ref{saddle}).
 For linear
system (\ref{saddle}), the suggestion of Cao \cite{Cao2019}  is a shift-splitting preconditioner $\mathcal{P}_{SS}$ and a relaxed version of the shift-splitting preconditioner $\mathcal{P}_{RSS}$ which are structured as follows:
\begin{equation}
\mathcal{P}_{SS} = \frac{1}{2} \begin{bmatrix} \alpha I + A & B^T & 0 \\ -B & \alpha I  &   - C^T\\ 0 & C & \alpha I \end{bmatrix},\hspace{0.3cm}\mathcal{P}_{RSS} = \frac{1}{2} \begin{bmatrix}  A & B^T & 0 \\ -B & \alpha I  &   - C^T\\ 0 & C & \alpha I \end{bmatrix}.
\end{equation}
The remainder of this paper is organized as follows. In Section $2$, we introduce a new preconditioning technique for the double saddle-point problem and analyze the spectral properties of the preconditioned system. Section $3$ presents numerical experiments that illustrate the effectiveness of the proposed method. Finally, conclusions and directions for future work are discussed in Section $4$.
\section{Preconditioning}
In this section, we study a regularization-based preconditioner for solving double  saddle-point problem~(\ref{saddle}), where $A \in \mathbb{R}^{n \times n}$ is a symmetric, large, sparse, and invertible matrix, and $B \in \mathbb{R}^{m \times n}$,  $C \in \mathbb{R}^{l \times m}$  are full row rank matrices. The central idea of preconditioning is to transform the linear system~(\ref{saddle})  into an equivalent system that is easier to solve. 
Applying left preconditioning to system (\ref{saddle}) yields the following transformed linear system:
\[
\mathcal{P}^{-1} \mathcal{A} \mathrm{w} = \mathcal{P}^{-1} b.
\]
Although right preconditioning is applicable in our context, this work focuses exclusively on left preconditioning. By taking advantage of the specific block structure of system  (\ref{saddle}), various block preconditioners have been proposed in the literature~\cite{Cao2019}. We introduce and analyze the matrix block preconditioner for (\ref{saddle}) defined by
\begin{equation}
\mathcal{P}_{R}=\left( \begin{array}{ccc}
	A & B^{T} & 0  \\ -B & S  & 0\\ 0  & C & \alpha I 
	\end{array} \right), \mathcal{P}_{RD}=\left( \begin{array}{ccc}
	A & B^{T} & 0  \\ B & \hat{S}  & 0\\ 0  & 0 & \alpha I 
	\end{array} \right) 
\label{Pr}
\end{equation}
with $\alpha > 0$ and $S=BA^{-1}B^{T}$ a symmetric and positive definite matrix, and  $\hat{S}$ its approximate.
\subsection{BLOCK PRECONDITIONER $\mathcal{P}_{R}$  AND EIGENVALUE ANALYSIS}
Let us consider the preconditioners in equation~(\ref{Pr}). We observe that these preconditioners are nonsingular since $A$ is symmetric positive definite and both $B$ and $C$ have full row rank. In the following, the eigenvalues of the preconditioned matrices corresponding to the proposed preconditioners are determined. 
\begin{theorem}
	\label{2p} 
	Let the preconditioner $\mathcal{P}_{R}$ be defined as in (\ref{Pr}). Then the matrix $\mathcal{P}^{-1}_{R}\mathcal{A}$ is diagonalizable and has $l + 2$ distinct eigenvalues $\{1, \frac{1}{2},\lambda_{1}  \dots,\lambda_{l}\}$.
\end{theorem}
\begin{proof}
 Let   $\lambda$ and $(x; y;z)$ be  an eigenvalue and its corresponding eigenvector of  $\mathcal{P}_{R}^{-1}\mathcal{A}$. According to the relationship $\mathcal{A}\mathrm{w}= \lambda \mathcal{P}_{R}\mathrm{w}$ we have
\begin{equation}
\label{eq}
\left\{
\begin{aligned}
(1-\lambda)Ax+ (1 - \lambda)B^{T}y&=0, \\
(\lambda-1)Bx-C^{T}z=\lambda Sy,\\
(1-\lambda)Cy=\alpha\lambda z.
\end{aligned}
\right.
\end{equation}
If $\lambda=1$, it gives 

\begin{equation}
\label{eqlambda=1}
\left\{
\begin{aligned}
-C^{T}z=Sy,\\
\alpha z=0.
\end{aligned}
\right.
\end{equation}
Since $\alpha >0$, we have $z=0$. Since $S$ is symmetric positive definite matrix, then from the first equation of (\ref{eqlambda=1}) , we get $y=0$.  Thus, there are $n$ linearly independent eigenvectors $(x^{(i)};\hspace{0.1cm} 0;\hspace{0.1cm}0)$, $i = 1, \dots, n$, corresponding to the eigenvalue $1$, where the vectors $x^{(i)}$ are arbitrary and linearly independent. If $\lambda\neq 1$ and $z=0$, after simple algebraic calculations $y$ satisfies the following equation:
\begin{equation}
B A^{-1} B^T y = 2\lambda B A^{-1} B^T y,
\label{Seig}
\end{equation}
we deduce that
\[
\lambda = \frac{1}{2}.\]
Then,  there are \( m \) linearly independent eigenvectors of the form 
\[
(-A^{-1}y^{(i)};y^{(i)};0), \quad i = 1, \dots, m,
\]
where \( y^{(i)} \) are arbitrary linearly independent vectors. 
If $\lambda\neq \frac{1}{2}$ and  $z\neq 0$, then from the first and second equations of (\ref{eq}), we get 
\begin{eqnarray}
\label{valC}
CS^{-1} C^T z=\eta z, 
\end{eqnarray}
with $\eta = \dfrac{\alpha\lambda(2\lambda-1 )}{(\lambda-1 )}$.
Since \( C S^{-1} C^T \) is a positive definite matrix, there are \( l \) linearly independent eigenvectors of the form 
\( z^{(i)} \), are orthogonal eigenvectors of  $C S^{-1} C^T$ . 
Let \( U \), \( V \), and \( W \) be the matrices whose columns are
\[
X = \begin{bmatrix} x^{(1)} & \cdots & x^{(n)} \end{bmatrix}, \quad
U = \begin{bmatrix} u^{(1)} & \cdots & u^{(m)} \end{bmatrix}, \quad
Y = \begin{bmatrix} {y}^{(1)} & \cdots & {y}^{(m)} \end{bmatrix},
\quad
Z = \begin{bmatrix} {z}^{(1)} & \cdots & {z}^{(l)} \end{bmatrix},
\]
\[
V = \begin{bmatrix} {v}^{(1)} & \cdots & {v}^{(l)} \end{bmatrix},\quad
W = \begin{bmatrix} {w}^{(1)} & \cdots & {w}^{(l)} \end{bmatrix}
.\]
where $u^{(i)}=-A^{-1}B^{T}y^{(i)}$,  $w^{(i)}=((1-\lambda)BA^{-1}B^{T}+\lambda S)^{-1}z^{(i)}$ and $v^{(i)}=-A^{-1}B^{T}w^{(i)}$.
Then 
\[
\mathcal{Z} = \begin{bmatrix} X & U &  V\\ 0 & Y  & W\\ 0 & 0  &Z\end{bmatrix}
\]
is a matrix of eigenvectors of \( \mathcal
{P}_{R}^{-1} \mathcal{A} \). Moreover, since \( X \), \( Y \) and \( Z \) are nonsingular matrices, \( \mathcal{Z} \) is also nonsingular. Consequently, \( \mathcal
{P}_{R}^{-1} \mathcal{A}\) is diagonalizable.
    
\end{proof}
\begin{theorem}    
Let the conditions of Theorem 2.1 be satisfied. Then all eigenvalues of $\mathcal{P}_{R}^{-1}\mathcal{A}$ are real and
\[
\sigma(\mathcal{P}_{R}^{-1}\mathcal{A}) \subset \{1\} \cup \{\frac{1}{2}\}\cup \left[\frac{1}{4}(\lambda_{min}(S_{C})+1)\sqrt{1-\displaystyle\frac{8\lambda_{min}(S_{C})}{\alpha(\lambda_{max}(S_{C})+1)}}, \displaystyle\frac{1}{4}(\lambda_{max}(S_{C})+1)\sqrt{1-\displaystyle\frac{8\lambda_{max}(S_{C})}{\alpha(\lambda_{min}(S_{C})+1)}} \right] ,
\]
where   $S_{C}=C S^{-1}C^T$ ,
$\lambda_{\max}(\cdot)$ and $\lambda_{\min}(\cdot)$ denote the maximum and minimum eigenvalues of a matrix, and $\sigma(\cdot)$ represents the set of all eigenvalues of a matrix.
\end{theorem} 
\medskip
\begin{proof}
Suppose that ${\lambda}$ is an arbitrary eigenvalue of $\mathcal{P}_{R}^{-1}\mathcal{A}$. From previous discussions, 
\[
\lambda=1, \quad \lambda=\frac{1}{2} \quad \text{or} \quad\eta=\dfrac{\alpha\lambda(2\lambda-1 )}{(\lambda-1 )}=\frac{s^*(C S^{-1} C^T) s}{s^{*}s}=\frac{s^*S_{C} s}{s^{*}s}, 
\]
for some $s \neq 0$. 
Hence, when $\lambda \neq 1$ and $\lambda \neq \frac{1}{2}$ we obtain
\begin{eqnarray}
\label{Ineq}
\displaystyle\lambda_{\min}(S_{C}) \displaystyle\leq \eta=\dfrac{\alpha\lambda(2\lambda-1 )}{(\lambda-1 )}\displaystyle\leq \lambda_{\max}(S_{C}),
\end{eqnarray}
$\displaystyle \eta=\dfrac{\alpha\lambda(2\lambda-1 )}{(\lambda-1 )}$ can be written as follows: 
\begin{eqnarray}
\label{eq3}
    2\alpha\lambda^{2}-\alpha(\eta+1)\lambda+\eta=0.
\end{eqnarray}
Notice that $\alpha \geq 2$. As a result, it is immediate to see that the roots of (\ref{eq3}) are real and given by
\[
\lambda_1 = \frac{1}{4}(\eta+1)\left(1+\sqrt{1-\displaystyle\frac{8\alpha\eta}{\alpha^2(\eta+1)}}\right), \quad \lambda_2 = \frac{1}{4}(\eta+1)\left(1-\sqrt{1-\displaystyle\frac{8\alpha\eta}{\alpha^2(\eta+1)}}\right).
\]
By using (\ref{Ineq}), it is not difficult to see that:
\begin{eqnarray}
\frac{1}{4}(\lambda_{min}(S_{C})+1)\left(1+\sqrt{1-\displaystyle\frac{8\lambda_{min}(S_{C})}{\alpha(\lambda_{max}(S_{C})+1)}}\right)\leq \lambda_1 &\leq& \frac{1}{4}(\lambda_{max}(S_{C})+1)\left(1+\sqrt{1-\displaystyle\frac{8\lambda_{max}(S_{C})}{\alpha(\lambda_{min}(S_{C})+1)}}\right)\nonumber,
\end{eqnarray}
and 
\begin{eqnarray}
\frac{1}{4}(\lambda_{min}(S_{C})+1)\left(1-\sqrt{1-\displaystyle\frac{8\lambda_{min}(S_{C})}{\alpha(\lambda_{max}(S_{C})+1)}}\right)\leq \lambda_2 &\leq& \frac{1}{4}(\lambda_{max}(S_{C})+1)\left(1-\sqrt{1-\displaystyle\frac{8\lambda_{max}(S_{C})}{\alpha(\lambda_{min}(S_{C})+1)}}\right)\nonumber,
\end{eqnarray}
then the proof is completed.
\end{proof}
   \textbf{Remark}
\begin{enumerate}
    \item When the hypotheses of the preceding theorem are satisfied,  we can see that the bound depends only on the distribution of the eigenvalues $\{1, \frac{1}{2},\lambda_1, \lambda_{2}\}$. 
    \item If $\lambda$ satisfies the  eigenvalue problem (\ref{valC}), then when \\ $\alpha \to \infty $, $\displaystyle\frac{  (\lambda_{min}(S_{C})+1)}{2}\leq \lambda_{1} \leq \displaystyle\frac{  (\lambda_{max}(S_{C})+1)}{2}$.
 \item    
    If  $\lambda$ satisfies the  eigenvalue problem (\ref{valC}), then
     when $\alpha \to \infty $, $\lambda_{2}\to 0$.
\end{enumerate}
\subsection{Algorithmic implementation  of the regularized
preconditioner  $\mathcal{P}_{R}$.}
In this part, we display the algorithmic implementation of $\mathcal{P}_{R}$,
in which inside Krylov subspace methods, the SPD subsystems were solved inexactly by the preconditioned conjugate gradient (PCG) method using loose tolerances. More precisely, the inner PCG method for linear systems with coefficient matrix $A$, $A + B^{T}{S}^{-1} B$ was terminated when the relative residual norm was below $10^{-6}$, when the maximum number of $100$ iterations was reached. The preconditioner for the PCG method is incomplete Cholesky factorizations constructed using the  function \texttt{ichol(., opts)} where \text{opts.type = 'ict'} with drop tolerance $10^{-2}$. Every step of the Krylov subspace
method, such as the GMRES method, is used in combination with the proposed preconditioner to solve the double saddle-point
problem $(\ref{saddle})$.
We summarize the implementation of preconditioners 
$\mathcal{P}_{R}$  in
Algorithms $1$.
For the linear systems corresponding to $A+  B^T {S}^{-1} B$.
\begin{itemize}
\label{approI}
    \item  Since $A+B^T \hat{S}^{-1} B$ 
    is  SPD matrix, we solve the linear systems corresponding to this  matrix independently
by the preconditioned conjugate gradient (PCG) method, the matrix is  formed  inside PCG  with incomplete Cholesky preconditioning, ichol(A).
\end{itemize}
As a result, we summarize the implementation of the preconditioners $\mathcal{P}_{R}$ 
in the form of the following algorithm:
 \begin{algorithm}[H]
 \label{algo1}
\begin{algorithmic}
\caption{: Computation of $(x; y;z) = \mathcal{P}_{R}^{-1} (r_1; r_2; r_3 )$} 
 \State Step 1. Solve $(A+B^T \hat{S}^{-1} B)x = r_1-B^{T}\hat{S}^{-1}g$ for $x$, where $\hat{S}$ is a diagonal matrix;
  \State Step 2. Solve $y =\hat{S}^{-1}(Bx+r_2)$ for $y$;
   \State  Step 3. Set $z=\displaystyle\frac{1}{\alpha} (r_3-Cy)$.
\end{algorithmic}
\end{algorithm}
\section{Numerical Experiments}
In this section, we use the following numerical examples to assess the relative efficiencies of the preconditioner $\mathcal{P}_{R}$  described in Section $2$. The programs are performed on a computer with an Intel Core i7-10750H CPU @ 2.60 GHz processor and 16.0 GB RAM using MATLAB R2020b. A natural choice of the matrix $\hat{S}$ is $\hat{S}=\alpha I+\displaystyle\frac{1}{\alpha }C^{T}C$. The parameter of the preconditioner $\mathcal{P}_{RSS}$ is taken as $\alpha =0.01$.
In the related tables, we denote by "Iter", the total required number of outer GMRES iterations  and we denote by "CPU"
the computational time in unit of second. The total number of inner GMRES (PCG) and FGMRES (PCG) iterations to solve sub-systems with coefficient matrix $A+ \displaystyle B^T \hat{S}^{-1} B$  are reported under $\text{Iter}_{\text{pcg}}$. No restart is used for either  GMRES  and FGMRES iterations. The initial guess is taken to be the zero vector and the iterations are stopped as soon as:
\[
\frac{\| {b}_{-}-\mathcal{A}_{-} \textbf{w}_k  \|_2}{\|{b}_{-} \|_2} \leq tol,
\]
where \( \textbf{w}_k \) is the computed \( k \)-th approximate solution. In the tables, we also include the relative error, relative residual, 
\begin{itemize}
    \item \[
\text{Err} := \frac{\| \textbf{w}_k - \textbf{w}^* \|_2}{\| \textbf{w}^* \|_2},
\]
\item $\mathrm{Res}$: norm of absolute relative residual is defined as :
\[
\text{Res} := 
\frac{\|{b}-\mathcal{A} \textbf{w}_k \|_2}{\|{b}\|_{2}},
\]

\end{itemize}
where \( \textbf{w}^* \) and \( \textbf{w}_k \) are respectively, the exact solution and its approximation obtained in the \( k \)-th iterate. 
\begin{exe}
As the first example, we consider the linear system of equations~(\ref{saddle}), where the block matrices are defined as
\[
A =
\begin{bmatrix}
I \otimes T + T \otimes I & 0 \\
0 & I \otimes T + T \otimes I
\end{bmatrix} \in \mathbb{R}^{2p^2 \times 2p^2}, \quad
B =
\begin{bmatrix}
I \otimes F \\
F \otimes I
\end{bmatrix} \in \mathbb{R}^{p^2 \times 2p^2}, \quad
C = E \otimes F \in \mathbb{R}^{p^2 \times p^2},
\]
where the symbol $\otimes$ denotes the Kronecker product, and $h = \frac{1}{p+1}$ is the discretization size. The matrices $T$, $F$, and $E$ are given by
\[
T = \frac{\nu}{h^2} \cdot \text{tridiag}(-1, 2, -1) \in \mathbb{R}^{p \times p}, \quad
F = \frac{1}{h} \cdot \text{tridiag}(0, 1, -1) \in \mathbb{R}^{p \times p}, \quad
E = \text{diag}(1, p+1, \ldots, p^2 - p + 1).
\]
\end{exe}
In following  Tables, we list numerical results with respect to Iter, $\mathrm{CPU}$,  Res and Err. To implement the $\mathcal{P}_{R}$ preconditioner efficiently, we need to choose the parameters    $\alpha$ appropriately since the analytic determination of the parameters which results in the fastest convergence of the preconditioned GMRES and FGMRES iterations appears to be quite a difficult problem. The three figures 1, 2 and 3, illustrate the CPU times of the preconditioned GMRES (PGMRES) method with different preconditioners  $\mathcal{P}_{BD}$ ,  $\mathcal{P}_{SS}$ ,  $\mathcal{P}_{RSS}$ , and  $\mathcal{P}_{R}$  for varying grid sizes and viscosities $(\nu=1,0.1,0.01)$. The x-axis represents the mesh resolution (grid size), while the y-axis reports CPU time in seconds.
\begin{table}[H]
\centering
\caption{Numerical results of the $\mathcal{P}$GMRES methods for $\nu=1$ $(tol=1e-12)$}
		\begin{tabular}{ |p{2.8cm}|p{0.9cm}|p{1.7cm}||p{1.7cm}||p{1.5cm}|p{1.5cm}|p{1.5cm}| }
			\hline
			Grid parameter &    &  $\mathcal{P}_{BD}$&$\mathcal{P}_{SS}$ & $\mathcal{P}_{RSS}$& $\mathcal{P}_{R}$    \\
			\hline
			$16\times 16$&Iter $\text{Iter}_{pcg}$ CPU Err Res   
    &  42\hspace{1.6cm}10\hspace{1.1cm} 0.20\hspace{1.6cm} 2.58e-08 \hspace{0.2cm}   4.30e-13&  13\hspace{1.6cm}7\hspace{1.1cm} 0.084\hspace{1.6cm} 1.58e-08 \hspace{0.2cm}   4.30e-14
			& 12\hspace{1.6cm}7\hspace{1.1cm} 0.08\hspace{1.6cm} 6.58e-09 \hspace{0.2cm}   7.30e-14& 8\hspace{1.6cm}6\hspace{1.1cm} 0.01\hspace{1.6cm} 1.58e-08 \hspace{0.2cm}   4.30e-13    
			\\
			\hline 
			$32\times 32$& Iter $\text{Iter}_{pcg}$ CPU Err Res & 53\hspace{1.6cm}10\hspace{1.1cm} 1.52\hspace{1.6cm}  7.64e-07\hspace{0.2cm}   1.22e-13& 
			13\hspace{1.6cm}7\hspace{1.1cm} 0.12\hspace{1.6cm}  6.64e-08\hspace{0.2cm}   3.22e-14 & 12\hspace{1.6cm}7\hspace{1.1cm} 0.09\hspace{1.6cm}  3.64e-08\hspace{0.2cm}   5.27e-14    &8\hspace{1.8cm}6\hspace{1.1cm} 0.04 \hspace{1.8cm}3.78e-06\hspace{0.8cm}4.30e-13  
			\\
			\hline
	$64\times 64$&Iter $\text{Iter}_{pcg}$ CPU Err Res    &  56\hspace{1.8cm}10\hspace{1.1cm}  8.53 \hspace{0.8cm}3.13e-06\hspace{0.8cm} 1.98e-13 & 13\hspace{1.8cm}7\hspace{1.1cm}  0.28 \hspace{0.8cm}2.83e-06\hspace{0.8cm} 6.98e-14& 12\hspace{1.8cm}7\hspace{1.1cm}  0.22 \hspace{0.8cm}2.83e-06\hspace{0.8cm} 6.98e-14&8\hspace{1.8cm}6\hspace{1.1cm}  0.16 \hspace{0.8cm}1.60e-05 \hspace{0.8cm} 2.04e-13 
\\
\hline
    $132\times 132$&Iter $\text{Iter}_{pcg}$ CPU Err Res    &  61\hspace{1.8cm}10\hspace{1.1cm} 18.90 \hspace{0.8cm}1.14e-05 \hspace{0.8cm} 6.12e-13 & 13\hspace{1.8cm}7\hspace{1.1cm} 2.90 \hspace{0.8cm}7.94e-06 \hspace{0.8cm} 7.10e-14& 12\hspace{1.8cm}7\hspace{1.1cm} 2.11 \hspace{0.8cm}7.94e-06 \hspace{0.8cm} 7.10e-14&8\hspace{1.8cm}6\hspace{1.1cm} 1.18 \hspace{0.8cm}1.73e-04 \hspace{0.8cm} 6.88e-13 
    \\
    \hline
       $160\times 160$&Iter $\text{Iter}_{pcg}$ CPU Err Res      & 64\hspace{1.8cm}7\hspace{1.1cm} 20.12 \hspace{0.8cm}7.94e-06 \hspace{0.8cm} 7.10e-13& 13\hspace{1.8cm}7\hspace{1.1cm} 4.11 \hspace{0.8cm}7.94e-06 \hspace{0.8cm} 7.10e-13& 12\hspace{1.8cm}7\hspace{1.1cm} 3.82 \hspace{0.8cm}3.56e-03 \hspace{0.8cm} 6.74e-13&8\hspace{1.8cm}6\hspace{1.1cm} 1.93 \hspace{0.8cm}3.56e-04\hspace{0.8cm} 2.43e-13 
       \\
     \hline
    
		\end{tabular}

\label{tab:v=1,GMRES}
\end{table}

\begin{table}[H]
\centering
\caption{Numerical results of the $\mathcal{P}$GMRES methods for $\nu=0.1$ $(tol=1e-12)$}
		\begin{tabular}{ |p{2.8cm}|p{0.9cm}||p{1.7cm}||p{1.5cm}|p{1.5cm}|p{1.5cm}| }
			\hline
			Grid parameter &   &$\mathcal{P}_{BD}$&$\mathcal{P}_{SS}$ & $\mathcal{P}_{RSS}$& $\mathcal{P}_{R}$    \\
			\hline
			$16\times 16$&Iter $\text{Iter}_{pcg}$ CPU Err Res   
    &  47\hspace{1.6cm}13\hspace{1.6cm}0.24\hspace{1.6cm} 6.70e-06 \hspace{0.2cm}   9.80e-13&  
		19\hspace{1.6cm}10\hspace{1.6cm}0.04\hspace{1.6cm} 5.70e-06 \hspace{0.2cm}   2.80e-13	& 18\hspace{1.6cm}10\hspace{1.6cm}0.03\hspace{1.6cm} 1.90e-06 \hspace{0.2cm}   1.50e-13 & 12\hspace{1.6cm}8\hspace{1.6cm}0.02\hspace{1.6cm} 4.86e-08 \hspace{0.2cm}   2.41e-13    
			\\
			\hline 
			$32\times 32$&Iter $\text{Iter}_{pcg}$ CPU Err Res         &58\hspace{1.8cm}13\hspace{1.6cm} 1.69 \hspace{1.8cm}8.23e-05\hspace{0.8cm}4.21e-13 & 
		19\hspace{1.8cm}10\hspace{1.6cm} 0.16 \hspace{1.8cm}7.22e-05\hspace{0.8cm}2.21e-13 	&  18\hspace{1.8cm}10\hspace{1.6cm} 0.1 \hspace{1.8cm}3.28e-05\hspace{0.8cm}1.01e-13 &12\hspace{1.8cm}8\hspace{1.6cm} 0.05 \hspace{1.8cm}7.74e-04\hspace{0.8cm}9.62e-13  
			\\
			\hline
	$64\times 64$&Iter $\text{Iter}_{pcg}$ CPU Err Res     & 61\hspace{1.8cm}13\hspace{1.6cm}  9.14 \hspace{0.8cm}9.29e-03 \hspace{0.8cm} 9.22e-13 &19\hspace{1.8cm}10\hspace{1.6cm}  0.54 \hspace{0.8cm}8.29e-04 \hspace{0.8cm} 9.20e-14  & 18\hspace{1.8cm}10\hspace{1.6cm}  0.48 \hspace{0.8cm}7.29e-04 \hspace{0.8cm} 6.76e-14 &12\hspace{1.8cm}8\hspace{1.6cm}  0.30 \hspace{0.8cm}2.81e-03 \hspace{0.8cm} 2.22e-13  
\\
\hline
    $132\times 132$&Iter $\text{Iter}_{pcg}$ CPU Err Rres      &67\hspace{1.8cm}13\hspace{1.6cm} 23.10 \hspace{0.8cm}5.21e+00   \hspace{0.8cm} 9.78e-13 & 18\hspace{1.8cm}10\hspace{1.6cm} 4.48 \hspace{0.8cm}3.21e+00   \hspace{0.8cm} 7.78e-13& 17\hspace{1.8cm}10\hspace{1.6cm} 3.48 \hspace{0.8cm}1.23e+00   \hspace{0.8cm} 6.75e-13 &10\hspace{1.8cm}8\hspace{1.6cm} 1.90 \hspace{0.8cm}1.73e-04 \hspace{0.8cm} 6.88e-13 
    \\
    \hline
       $160\times 160$&Iter $\text{Iter}_{pcg}$ CPU Err Res      & 70\hspace{1.8cm}13\hspace{1.6cm} 26.20 \hspace{0.8cm}5.21e+00   \hspace{0.8cm} 9.88e-13& 18\hspace{1.8cm}10\hspace{1.6cm} 6.83 \hspace{0.8cm}3.41e+00 \hspace{0.8cm} 7.31e-13& 17\hspace{1.8cm}10\hspace{1.6cm} 5.83 \hspace{0.8cm}2.41e+00 \hspace{0.8cm} 5.37e-13&9\hspace{1.8cm}8\hspace{1.6cm} 2.50 \hspace{0.8cm}1.37e-03 \hspace{0.8cm} 5.17e-13 
       \\
     \hline
    
		\end{tabular}

\label{tab:v=0.1,GMRES}
\end{table}

\begin{table}[H]
\centering
\caption{Numerical results of the $\mathcal{P}$GMRES methods for $\nu=0.01$ $(tol=1e-12)$}
		\begin{tabular}{ |p{2.8cm}|p{0.9cm}|p{1.7cm}||p{1.5cm}|p{1.5cm}|p{1.5cm}| }
			\hline
			Grid parameter &   &$\mathcal{P}_{BD}$&$\mathcal{P}_{SS}$ & $\mathcal{P}_{RSS}$& $\mathcal{P}_{R}$   \\
			\hline
			$16\times 16$&Iter $\text{Iter}_{pcg}$ CPU Err Rres   
    &  61\hspace{1.6cm}19\hspace{1.6cm}0.34\hspace{1.6cm} 8.16e-06 \hspace{0.2cm}   4.85e-13&  
35\hspace{1.6cm}15\hspace{1.6cm}0.09\hspace{1.6cm} 3.16e-06 \hspace{0.2cm}   4.85e-14	& 34\hspace{1.6cm}15\hspace{1.6cm}0.07\hspace{1.6cm} 1.56e-06 \hspace{0.2cm}   6.35e-14    & 16\hspace{1.6cm}15\hspace{1.6cm}0.03\hspace{1.6cm} 1.55e-07 \hspace{0.2cm}   6.41e-13    
			\\
			\hline 
			$32\times 32$&Iter $\text{Iter}_{pcg}$ CPU Err Res         & 68\hspace{1.6cm}15\hspace{1.6cm}0.46\hspace{1.6cm} 1.16e-04 \hspace{0.2cm}   7.87e-13&   34\hspace{1.8cm}19\hspace{1.6cm} 0.36 \hspace{1.8cm}4.10e-02\hspace{0.8cm}7.41e-13 &33\hspace{1.8cm}15\hspace{1.6cm} 0.26 \hspace{1.8cm}3.10e-02\hspace{0.8cm}6.41e-13 
			   &15\hspace{1.8cm}13\hspace{1.6cm} 0.11 \hspace{1.8cm}2.68e-05\hspace{0.8cm}4.07e-13 
			\\
			\hline
	$64\times 64$&Iter $\text{Iter}_{pcg}$ CPU Err Res      &71\hspace{1.8cm} 16\hspace{1.6cm}10.55 \hspace{0.8cm}1.84e+01  \hspace{0.8cm} 7.60e-13 &33\hspace{1.8cm}15\hspace{1.6cm} 1.63\hspace{0.8cm}4.21e+00 \hspace{0.8cm}1.12e-12  & 32\hspace{1.8cm}15\hspace{1.6cm} 1.13 \hspace{0.8cm}2.61e+00 \hspace{0.8cm}8.92e-13 &14\hspace{1.8cm}13\hspace{1.6cm} 0.56 \hspace{0.8cm}4.99e-04 \hspace{0.8cm} 5.59e-13 
\\
\hline
    $132\times 132$&Iter $\text{Iter}_{pcg}$ CPU Err Res      & 81\hspace{1.8cm} 16\hspace{1.6cm}33.50 \hspace{0.8cm}4.94e+01  \hspace{0.8cm} 7.60e-13& 34\hspace{1.8cm}15\hspace{1.6cm} 10.22 \hspace{0.8cm}3.26e+01 \hspace{0.8cm}  8.46e-13&33\hspace{1.8cm}15\hspace{1.6cm} 9.32 \hspace{0.8cm}1.06e+01 \hspace{0.8cm}  6.66e-13 &13\hspace{1.8cm}13\hspace{1.6cm} 2.73 \hspace{0.8cm}4.15e-02 \hspace{0.8cm} 1.34e-13 
    \\
    \hline
       $160\times 160$&Iter $\text{Iter}_{pcg}$ CPU Err Res      & 92\hspace{1.8cm} 16\hspace{1.6cm}38.50 \hspace{0.8cm}5.94e+01  \hspace{0.8cm} 8.60e-13& 34\hspace{1.8cm}15\hspace{1.6cm} 16.38 \hspace{0.8cm} 8.21e+00 \hspace{0.8cm}  7.87e-12 &33\hspace{1.8cm}15\hspace{1.6cm} 15.28 \hspace{0.8cm} 7.26e+00 \hspace{0.8cm}  6.87e-13 &12\hspace{1.8cm}13\hspace{1.6cm} 4.46 \hspace{0.8cm}1.37e-02 \hspace{0.8cm} 5.17e-13 
       \\
     \hline
    
		\end{tabular}

\label{tab:v=0.01,GMRES}
\end{table}

In Tables above we present the numbers of GMRES iterations and CPU times required to solve the system (\ref{saddle})  to a tolerance of $1e-12$, for each problem setup described above. We note that for  $p=32$, $p=64$, $p=132$
and $p=160$, the linear systems have dimensions of (respectively) $12288$,  $52272$, and
$76800$.
According to Tables~\ref{tab:v=1,GMRES}, \ref{tab:v=0.1,GMRES} and \ref{tab:v=0.01,GMRES}, it can be observed that the  regularized preconditioner $\mathcal{P}_{R}$
 requires lowest $\mathrm{Iter}$ and CPU time, which implies that the  $\mathcal{P}_{R}$GMRES method is superior to the  $\mathcal{P}_{BD}$GMRES, $\mathcal{P}_{SS}$GMRES and $\mathcal{P}_{RSS}$GMRES
 methods in terms of computing efficiency. This is evident as the $\mathcal{P}_{R}$GMRES requires less CPU time. The  $\mathcal{P}_{R}$  
 preconditioner  performs less effectively, particularly when the  viscosity $\nu$ decrease. The $\mathcal{P}_{R}$GMRES remains highly efficient even for larger  parameter $p$ and  for smalest values of $\nu$. The $\mathcal{P}_{BD}$, $\mathcal{P}_{SS}$, and $\mathcal{P}_{RSS}$ preconditioners perform less effectively,
particularly when $\nu$ parameter  decrease. It should be noted that if $\nu$ parameter is decreased beyond $0.01$, the $\mathcal{P}_{BD}$, $\mathcal{P}_{SS}$, and $\mathcal{P}_{RSS}$ preconditioners encounter limitations. This is due to the fact that the eigenvalues of the preconditioned matrix become highly clustered into many groups, which
negatively affects the convergence rate and the stability of the preconditioned GMRES method. With the use of our proposed preconditioner $\mathcal{P}_{R}$, we overcome this limitation and significantly improve the efficiency of the preconditioned GMRES, even for very small values of the parameter $\nu$.\\
 We list the numerical results of the  FGMRES methods with  various values of $\nu$ in Tables $\ref{tabcc:v=1,FGMRES}$, $\ref{tabcc:v=0.1,FGMRES}$ and $\ref{tabcc:v=0.01,FGMRES}$.\\
\begin{table}[H]
\centering
\caption{Numerical results of the preconditioned FGMRES methods for $\nu=1$ $(tol=1e-7)$}
		\begin{tabular}{ |p{2.8cm}|p{0.9cm}||p{1.7cm}||p{1.5cm}|p{1.5cm}|p{1.5cm}| }
			\hline
			Grid parameter &   &$\mathcal{P}_{BD}$&$\mathcal{P}_{SS}$ & $\mathcal{P}_{RSS}$& $\mathcal{P}_{R}$  \\
			\hline
			$16\times 16$&Iter $\text{Iter}_{pcg}$ CPU Err Res   
    &  40\hspace{1.6cm}10\hspace{1.6cm} 0.18\hspace{1.6cm}  3.15e-06 \hspace{0.2cm}     8.57e-08 & 11\hspace{1.6cm}7\hspace{1.6cm} 0.02\hspace{1.6cm}  4.50e-07 \hspace{0.2cm}     7.57e-08 &  10\hspace{1.6cm}7\hspace{1.6cm} 0.02\hspace{1.6cm}  1.50e-07 \hspace{0.2cm}     1.57e-08    & 10\hspace{1.6cm}6\hspace{1.6cm} 0.001\hspace{1.6cm}  4.60e-04 \hspace{0.2cm}      5.83e-08    
			\\
			\hline 
			$32\times 32$&Iter $\text{Iter}_{pcg}$ CPU Err Res      &  50\hspace{1.6cm}10\hspace{1.6cm} 1.50\hspace{1.6cm}  5.15e-06 \hspace{0.2cm}     5.57e-08 & 11\hspace{1.8cm}7\hspace{1.6cm}  0.18 \hspace{1.8cm}  3.68e-06\hspace{0.8cm}8.17e-08& 10\hspace{1.8cm}7\hspace{1.6cm}  0.09 \hspace{1.8cm}  1.78e-06\hspace{0.8cm}5.97e-08
			  &10\hspace{1.8cm}6\hspace{1.6cm}  0.03 \hspace{1.8cm} 1.28e-05\hspace{0.8cm}7.41e-08   
			\\
			\hline
	$64\times 64$&Iter $\text{Iter}_{pcg}$ CPU Err Res      & 54\hspace{1.6cm}10\hspace{1.6cm} 8.72\hspace{1.6cm}  4.15e-06 \hspace{0.2cm}     3.57e-08& 11\hspace{1.8cm}7\hspace{1.6cm}  0.14 \hspace{0.8cm}8.34e-05 \hspace{0.8cm}  8.17e-08& 10\hspace{1.8cm}7\hspace{1.6cm}  0.14 \hspace{0.8cm}7.30e-05 \hspace{0.8cm}  7.37e-08 &10\hspace{1.8cm}6\hspace{1.6cm}  0.12 \hspace{0.8cm}1.90e-05 \hspace{0.8cm}  1.96e-08  \\
\hline
    $132\times 132$&Iter $\text{Iter}_{pcg}$ CPU Err Res      & 58\hspace{1.8cm}10\hspace{1.6cm}  9.13 \hspace{0.8cm} 6.22e-04 \hspace{0.8cm}  4.11e-08& 11\hspace{1.8cm}7\hspace{1.6cm}  2.37 \hspace{0.8cm} 5.22e-04 \hspace{0.8cm}  6.14e-08& 10\hspace{1.8cm}7\hspace{1.6cm}  1.97 \hspace{0.8cm} 1.20e-04 \hspace{0.8cm}  4.54e-08
&10\hspace{1.8cm}6\hspace{1.6cm}  0.97 \hspace{0.8cm}1.05e-04 \hspace{0.8cm} 1.61e-08 
    \\
    \hline
       $160\times 160$&Iter $\text{Iter}_{pcg}$ CPU Err Res     & 61\hspace{1.8cm}10\hspace{1.6cm}  10.39 \hspace{0.8cm}2.12e-03 \hspace{0.8cm}    7.51e-08& 11\hspace{1.8cm}7\hspace{1.6cm}  3.39 \hspace{0.8cm}3.12e-03 \hspace{0.8cm}    9.51e-08& 10\hspace{1.8cm}7\hspace{1.6cm}  2.79 \hspace{0.8cm}1.92e-04 \hspace{0.8cm}    5.71e-08
 &10\hspace{1.8cm}6\hspace{1.6cm}  1.54 \hspace{0.8cm}1.36e-04 \hspace{0.8cm}  1.22e-08 
       \\
     \hline
    
		\end{tabular}

\label{tabcc:v=1,FGMRES}
\end{table}

\begin{table}[H]
\centering
\caption{Numerical results of the preconditioned FGMRES methods for $\nu=0.1$ $(tol=1e-7)$}
		\begin{tabular}{ |p{2.8cm}|p{0.9cm}||p{1.7cm}||p{1.5cm}|p{1.5cm}|p{1.5cm}| }
			\hline
			Grid parameter &  &$\mathcal{P}_{BD}$&$\mathcal{P}_{SS}$ & $\mathcal{P}_{RSS}$& $\mathcal{P}_{r}$   \\
			\hline
			$16\times 16$&Iter $\text{Iter}_{pcg}$ CPU Err Res   
    & 44\hspace{1.6cm} 12\hspace{1.8cm}0.20\hspace{1.6cm} 4.14e-06 \hspace{0.2cm}   7.13e-08  &16\hspace{1.6cm} 11\hspace{1.8cm}0.06\hspace{1.6cm} 7.14e-06 \hspace{0.2cm}   5.13e-08  & 15\hspace{1.6cm} 10\hspace{1.8cm}0.03\hspace{1.6cm} 3.04e-06 \hspace{0.2cm}   1.63e-08   &  13\hspace{1.6cm} 8\hspace{1.8cm}0.01\hspace{1.6cm} 4.73e-06 \hspace{0.2cm}   2.17e-08   
			\\
			\hline 
			$32\times 32$&Iter $\text{Iter}_{pcg}$ CPU Err Res         & 54\hspace{1.6cm} 12\hspace{1.8cm}1.58\hspace{1.6cm} 4.14e-06 \hspace{0.2cm}   7.13e-08 & 
		16\hspace{1.8cm}11\hspace{1.8cm} 0.09 \hspace{1.8cm}7.88e-05\hspace{0.8cm}3.91e-08  	&  15\hspace{1.8cm}10\hspace{1.8cm} 0.07 \hspace{1.8cm}6.33e-05\hspace{0.8cm}2.68e-08    &14\hspace{1.8cm}8\hspace{1.8cm} 0.03 \hspace{1.8cm}1.54e-04\hspace{0.8cm}4.62e-08  
			\\
			\hline
	$64\times 64$&Iter $\text{Iter}_{pcg}$ CPU Err Res      & 59\hspace{1.6cm} 12\hspace{1.8cm}9.18\hspace{1.6cm} 4.14e-06 \hspace{0.2cm}   7.13e-08 & 18\hspace{1.8cm}10\hspace{1.8cm} 0.70 \hspace{0.8cm}8.81e-04 \hspace{0.8cm} 9.98e-08 & 17\hspace{1.8cm}10\hspace{1.8cm} 0.38 \hspace{0.8cm}4.86e-04 \hspace{0.8cm} 9.94e-08  &14\hspace{1.8cm}8\hspace{1.8cm} 0.26 \hspace{0.8cm}2.98e-04 \hspace{0.8cm} 7.83e-08 
\\
\hline
    $132\times 132$&Iter $\text{Iter}_{pcg}$ CPU Err Res      & 71\hspace{1.6cm} 12\hspace{1.8cm}10.21\hspace{1.6cm} 4.14e-06 \hspace{0.2cm}   7.13e-08& 15\hspace{1.8cm}10\hspace{1.8cm} 2.34 \hspace{0.8cm}2.48e-03 \hspace{0.8cm} 5.12e-06& 15\hspace{1.8cm}10\hspace{1.8cm} 2.01 \hspace{0.8cm}1.38e-03 \hspace{0.8cm} 1.12e-06&14\hspace{1.8cm}8\hspace{1.8cm} 1.85 \hspace{0.8cm}9.22e-03 \hspace{0.8cm} 9.20e-08 
    \\
    \hline
       $160\times 160$&Iter $\text{Iter}_{pcg}$ CPU Err Res      &75\hspace{1.6cm} 12\hspace{1.8cm}11.12\hspace{1.6cm} 5.14e-06 \hspace{0.2cm}   8.13e-08 & 16\hspace{1.8cm}10\hspace{1.8cm} 3.95 \hspace{0.8cm}3.81e-03 \hspace{0.8cm} 6.26e-08& 15\hspace{1.8cm}10\hspace{1.8cm} 3.24 \hspace{0.8cm}1.71e-03 \hspace{0.8cm} 4.33e-08 &14\hspace{1.8cm}10\hspace{1.8cm} 3.14 \hspace{0.8cm}3.81e-03 \hspace{0.8cm} 8.69e-08 
       \\
     \hline
    
		\end{tabular}

\label{tabcc:v=0.1,FGMRES}
\end{table}

The CPU time comparisons reported in Figures~1--3 clearly highlight the differences in efficiency and robustness of the tested preconditioners within the PGMRES framework. For all viscosity values and grid sizes, the  $\mathcal{P}_{BD}$ preconditioner exhibits the poorest performance, with CPU times growing steeply and reaching more than 20--35 seconds for the finest meshes, thus confirming its lack of scalability. The  $\mathcal{P}_{SS}$  and  $\mathcal{P}_{RSS}$  preconditioners show moderate improvements, but their CPU times increase significantly as the viscosity decreases, reflecting their sensitivity to the ill-conditioning of the system. In contrast, the  $\mathcal{P}_{R}$  preconditioner consistently outperforms all the others, maintaining very low CPU times across all grid sizes and viscosities. In particular, for the challenging case $\nu = 0.01$,  $\mathcal{P}_{R}$  remains below 5 seconds even on the largest grid, whereas the other methods degrade substantially. These results demonstrate that  $\mathcal{P}_{R}$  is the most robust and efficient preconditioning strategy, providing superior scalability with respect to both the mesh size and the viscosity parameter.
\begin{table}[H]
\centering
\caption{Numerical results of the preconditioned FGMRES methods for $\nu=0.01$ $(tol=1e-7)$}
		\begin{tabular}{ |p{2.8cm}|p{0.9cm}||p{1.7cm}||p{1.5cm}|p{1.5cm}|p{1.5cm}| }
			\hline
			Grid parameter &  &$\mathcal{P}_{BD}$&$\mathcal{P}_{SS}$ & $\mathcal{P}_{RSS}$& $\mathcal{P}_{r}$    \\
			\hline
			$16\times 16$&Iter $\text{Iter}_{pcg}$ CPU Err Res   
    & 48\hspace{1.6cm} 17\hspace{1.8cm}0.28\hspace{1.6cm} 4.14e-06 \hspace{0.2cm}   7.13e-08  &  
	31\hspace{1.6cm}16\hspace{1.6cm} 0.07\hspace{1.6cm}  4.99e-04 \hspace{0.2cm}   9.71e-08		& 30\hspace{1.6cm}15\hspace{1.6cm} 0.04\hspace{1.6cm}  2.99e-04 \hspace{0.2cm}   8.74e-08 & 22\hspace{1.6cm}13\hspace{1.6cm} 0.02\hspace{1.6cm} 19e-05 \hspace{0.2cm}   1.45e-08    
			\\
			\hline 
			$32\times 32$&Iter $\text{Iter}_{pcg}$ CPU Err Res         & 57\hspace{1.6cm} 19\hspace{1.8cm}1.78\hspace{1.6cm} 4.14e-06 \hspace{0.2cm}   7.13e-08& 
			31\hspace{1.6cm}18\hspace{1.6cm} 0.21\hspace{1.6cm}  3.90e-03  \hspace{0.2cm}   6.14e-08&   30\hspace{1.6cm}17\hspace{1.6cm} 0.13\hspace{1.6cm}  2.66e-03  \hspace{0.2cm}   5.84e-08&22\hspace{1.8cm}13\hspace{1.8cm} 0.08 \hspace{1.8cm}1.32e-03\hspace{0.8cm}6.69e-08  
			\\
			\hline
	$64\times 64$&Iter $\text{Iter}_{pcg}$ CPU Err Res      & 61\hspace{1.6cm} 20\hspace{1.8cm}9.90\hspace{1.6cm} 4.14e-06 \hspace{0.2cm}   7.13e-08& 32\hspace{1.6cm}18\hspace{1.6cm} 1.20\hspace{1.6cm}  3.61e-02   \hspace{0.2cm}   9.21e-08 & 31\hspace{1.6cm}17\hspace{1.6cm} 0.72\hspace{1.6cm}  1.51e-02   \hspace{0.2cm}   8.28e-08 &22\hspace{1.8cm}13\hspace{1.8cm} 0.64 \hspace{0.8cm}8.76-04 \hspace{0.8cm} 8.82e-08 
\\
\hline
$132\times 132$&Iter $\text{Iter}_{pcg}$ CPU Err Res      &75\hspace{1.8cm}20\hspace{1.8cm} 11.20 \hspace{0.8cm}4.80e-02 \hspace{0.8cm}    5.39e-08 & 42\hspace{1.8cm}16\hspace{1.8cm} 6.12 \hspace{0.8cm}4.80e-02 \hspace{0.8cm}    5.39e-08& 41\hspace{1.8cm}15\hspace{1.8cm} 5.52 \hspace{0.8cm}2.81e-02 \hspace{0.8cm}    3.30e-08
&24\hspace{1.8cm}13\hspace{1.8cm} 4.39 \hspace{0.8cm}5.45e-03 \hspace{0.8cm}   3.20e-08  
    \\
    \hline
       $160\times 160$&Iter $\text{Iter}_{pcg}$ CPU Err Res      & 78\hspace{1.8cm}20\hspace{1.8cm} 11.87 \hspace{0.8cm}4.80e-02 \hspace{0.8cm}    5.39e-08& 33\hspace{1.8cm}17\hspace{1.8cm} 10.12 \hspace{0.8cm}8.60e-02 \hspace{0.8cm} 6.99e-08& 32\hspace{1.8cm}16\hspace{1.8cm} 9.10 \hspace{0.8cm}7.67e-02 \hspace{0.8cm} 4.99e-08 &24\hspace{1.8cm}16\hspace{1.8cm} 7.46 \hspace{0.8cm}9.55e-03 \hspace{0.8cm} 6.65e-08 
       \\
     \hline

		\end{tabular}

\label{tabcc:v=0.01,FGMRES}
\end{table}

From these experiments, it is evident that  our $\mathcal{P}_{R}$FGMRES method, in all trials, it necessitates a smaller amount of $\mathrm{CPU}$ time compared to other FGMRES methods. The comparative analysis of simulation results in Tables  \ref{tabcc:v=1,FGMRES}, \ref{tabcc:v=0.1,FGMRES} and \ref{tabcc:v=0.01,FGMRES} indicates a substantial improvement in the performance of FGMRES, considering CPU times, Err and Res. This improvement is particularly notable when utilizing the  $\mathcal{P}_{R}$ preconditioner compared to the  $\mathcal{P}_{BD}$, $\mathcal{P}_{SS}$ and  $\mathcal{P}_{RSS}$ preconditioners. Furthermore, another observation which can be posed here is  that the $\mathcal{P}_{R}$  preconditioner is  sensitive to the parameter $\nu$. The $\mathcal{P}_{R}$  preconditioned FGMRES methods exhibit
worse convergence behaviors as $\nu$ decreases, and achieve better performance when  $\nu$ increases. 
To demonstrate the effectiveness of the preconditioner $\mathcal{P}_{R}$, Figs. 1–4 present the eigenvalue distributions of the original matrix $\mathcal{A}$ and the preconditioned matrix $\mathcal{P}_{R}^{-1}\mathcal{A}$ on a $32 \times 32$ grid, for $\nu = 1$ and $\nu = 0.01$. As shown in Figs. $1$ and $3$, the eigenvalues of $\mathcal{A}$ are widely scattered, whereas Figs. 2 and 4 reveal that the preconditioned matrix exhibits significantly more clustered eigenvalues. This clustering is more pronounced for $\mathcal{P}_{R}^{-1}\mathcal{A}$ than for $\mathcal{A}$. Notably, at $\nu = 0.01$, both matrices show tighter clustering than at $\nu = 1$. The numerical results in the tables confirm that increasing $\nu$ leads to improved performance, while decreasing $\nu$ reduces the preconditioner effectiveness. Since the convergence of Krylov subspace methods depends on the eigenvalue distribution, these observations explain the reduced efficiency of the preconditioner at lower viscosities, particularly when $\nu = 0.01$.
Therefore, the preconditioner $\mathcal{P}_{R}$ is more effective in solving the non-symmetric double  saddle-point systems (\ref{saddle}).

We used the preconditioned MINRES method to solve the symmetric  double saddle-point problems  given in (\ref{saddle}). The preconditioners considered are defined as follows:
\[
\mathcal{P}_{RD} = 
\begin{pmatrix}
A & B^T  & 0\\
B & \hat{S} &0\\
0 &  0   &\alpha I
\end{pmatrix}, 
\qquad
\mathcal{P}_{BD} = 
\begin{pmatrix}
A & 0  & 0\\
0 & \hat{S}  &0 \\
0 & 0 & C\hat{S}^{-1}C^T
\end{pmatrix}.
\]
In all the numerical experiments below,  we choose
a zero initial guess and a stopping tolerance of $1e-07$ based on the reduction in relative residual norms for all Krylov subspace solvers tested unless otherwise indicated.
For the preconditioned  MINRES, the iterations are terminated when the following relative residual condition is satisfied:
\[
\frac{\left\| \mathcal{P}^{-1}b_{+} - P^{-1}\mathcal{A}_{+}w^{(k)} \right\|}{\left\| \mathcal{P}^{-1}b_{+} \right\|} < 10^{-07},
\]
where $\mathcal{P}$ is one of the preconditioners $\mathcal{P}_{RD}$ or  $\mathcal{P}_{BD}$. The symbol $\| \cdot \|$ denotes the Frobenius norm, and $w^{(k)} \in \mathbb{R}^{n+m+l}$ represents the current iterate.
When using Algorithm~2 to solve the first system in Algorithm~2, the preconditioner employed is an incomplete Cholesky factorization with drop tolerance, computed using MATLAB’s built-in function \texttt{ichol(.,opts)}, where the options are set as follows:
\begin{itemize}
    \item \texttt{opts.type = 'ict'},
    \item \texttt{opts.droptol = 1e-2}.
\end{itemize}
The stopping criterion for the inner iterations is based on the relative residual norm being smaller than a tolerance $\texttt{tol} = 10^{-6}$.
The numerical results for the symmetric double saddle-point problem,  are given in Tables below. As observed in Tables $7$, $8$ and $9$, the $\mathcal{P}_{RD}$ preconditioned global MINRES method with the proper parameter $\alpha$ has a better performance than the
$\mathcal{P}_{BD}$ preconditioned  MINRES method in terms of the iterations and CPU times. The three figures 11, 12 and 13, display the CPU times of the preconditioned MINRES method
for different values of the viscosity parameter $\nu \in \{1, 0.1, 0.01\}$, with a fixed tolerance of $10^{-7}$.
Two preconditioning strategies are compared:  $\mathcal{P}_{BD}$  and $\mathcal{P}_{RD}$ 
\begin{table}[H]
\centering
\caption{Numerical results of the preconditioned MINRES methods for $\nu=1$ $(tol=1e-07)$}
		\begin{tabular}{ |p{2.8cm}|p{0.9cm}||p{1.7cm}||p{1.5cm}|}
			\hline
			Grid parameter &  & $\mathcal{P}_{BD}$& $\mathcal{P}_{RD}$  \\
			\hline
	$64\times 64$&Iter $\text{Iter}_{pcg}$ CPU Err Res     & 54\hspace{1.8cm}10\hspace{1.6cm}  8.72 \hspace{0.8cm}7.30e-04 \hspace{0.8cm}  7.37e-08 &23\hspace{1.8cm}6\hspace{1.6cm}  1.67 \hspace{0.8cm}1.90e-04 \hspace{0.8cm}  7.51e-09  \\
\hline
    $132\times 132$&Iter $\text{Iter}_{pcg}$ CPU Err Res     & 58\hspace{1.8cm}10\hspace{1.6cm}  9.13 \hspace{0.8cm} 1.20e-04 \hspace{0.8cm}  4.54e-08
&10\hspace{1.8cm}6\hspace{1.6cm}  4.52 \hspace{0.8cm}1.05e-04 \hspace{0.8cm} 1.61e-09 
    \\
    \hline
       $160\times 160$&Iter $\text{Iter}_{pcg}$ CPU Err Res    & 61\hspace{1.8cm}10\hspace{1.6cm}  10.39 \hspace{0.8cm}1.92e-03 \hspace{0.8cm}    5.71e-08
 &13\hspace{1.8cm}6\hspace{1.6cm}  8.52 \hspace{0.8cm}1.36e-03 \hspace{0.8cm}  1.61e-08 
       \\
     \hline
    
		\end{tabular}

\label{tabcc:v=1,MINRES}
\end{table}

\begin{table}[H]
\centering
\caption{Numerical results of the preconditioned MINRES methods for $\nu=0.1$ $(tol=1e-7)$}
			\begin{tabular}{ |p{2.8cm}|p{0.9cm}||p{1.7cm}||p{1.5cm}|}
			\hline
			Grid parameter &  & $\mathcal{P}_{BD}$& $\mathcal{P}_{RD}$  \\
			\hline
	$64\times 64$&Iter $\text{Iter}_{pcg}$ CPU Err Res     & 59\hspace{1.8cm}12\hspace{1.6cm}  9.18 \hspace{0.8cm}4.30e-05 \hspace{0.8cm}  7.13e-08 &31\hspace{1.8cm}8\hspace{1.6cm}  2.26 \hspace{0.8cm}2.56e-04 \hspace{0.8cm}  1.44e-09  \\
\hline
    $132\times 132$&Iter $\text{Iter}_{pcg}$ CPU Err Res     & 71\hspace{1.8cm}12\hspace{1.6cm}  10.21 \hspace{0.8cm} 1.20e-04 \hspace{0.8cm}  4.54e-08
&21\hspace{1.8cm}6\hspace{1.6cm}  9.67 \hspace{0.8cm}3.12e-03 \hspace{0.8cm} 1.97e-09 
    \\
    \hline
       $160\times 160$&Iter $\text{Iter}_{pcg}$ CPU Err Res    & 75\hspace{1.8cm}12\hspace{1.6cm}  20.79 \hspace{0.8cm}1.92e-04 \hspace{0.8cm}    5.71e-08
 &18\hspace{1.8cm}6\hspace{1.6cm}  16.40 \hspace{0.8cm}4.30e-03 \hspace{0.8cm}  1.49e-09 
       \\
     \hline
    
		\end{tabular}
\label{tabcc:v=0.1,MINRES}
\end{table}
\begin{table}[H]
\centering
\caption{Numerical results of the preconditioned MINRES methods for $\nu=0.01$ $(tol=1e-7)$}
				\begin{tabular}{ |p{2.8cm}|p{0.9cm}||p{1.7cm}||p{1.5cm}|}
			\hline
			Grid parameter &  & $\mathcal{P}_{BD}$& $\mathcal{P}_{RD}$  \\
			\hline
	$64\times 64$&Iter $\text{Iter}_{pcg}$ CPU Err Res     & 61\hspace{1.8cm}20\hspace{1.6cm}  9.90 \hspace{0.8cm}4.59e-04 \hspace{0.8cm}  9.58e-10 &34\hspace{1.8cm}13\hspace{1.6cm}  3.15 \hspace{0.8cm}4.59e-04 \hspace{0.8cm}  9.58e-10 \\
\hline
    $132\times 132$&Iter $\text{Iter}_{pcg}$ CPU Err Res     & 75\hspace{1.8cm}20\hspace{1.6cm}  19.97 \hspace{0.8cm} 1.20e-04 \hspace{0.8cm}  4.54e-08
&26\hspace{1.8cm}14\hspace{1.6cm}  14.81 \hspace{0.8cm}1.92e-02 \hspace{0.8cm} 6.12e-10 
    \\
    \hline
       $160\times 160$&Iter $\text{Iter}_{pcg}$ CPU Err Res    & 78\hspace{1.8cm}20\hspace{1.6cm}  27.79 \hspace{0.8cm}1.92e-04 \hspace{0.8cm}    5.71e-08
 &21\hspace{1.8cm}13\hspace{1.6cm}  23.49 \hspace{0.8cm}1.60e-02 \hspace{0.8cm}  4.46e-10 
       \\
     \hline
    
		\end{tabular}

\label{tabcc:v=0.01,MINRES}
\end{table}

Tables $7$, $8$ and $9$. From the numerical results in this table, we observe that the $\mathcal{P}_{RD}$
preconditioner
is superior to the $\mathcal{P}_{BD}$ preconditioner in terms of iterations and CPU times. Noticeably, our proposed  $\mathcal{P}_{RD}$ preconditioner has shown numerical robustness for a wide range of the parameter $\nu$,
\section{Conclusion and future studies
}\label{conclusion}
We have carried out a detailed spectral analysis of a family of double saddle-point linear systems, preconditioned by the non-symmetric preconditioner $\mathcal{P}_{R}$ proposed in (\ref{Pr}), within the framework of multiple saddle-point problems. We derived accurate bounds for both the negative and positive intervals containing the eigenvalues of the preconditioned matrix. Numerical results on synthetic test problems confirm that these bounds are close to the endpoints of the spectral intervals.
We also evaluated the performance of this preconditioner for various values of the viscosity parameter $\nu$. A careful selection of the block approximation $\hat{S}$ yields a highly efficient preconditioner, as demonstrated numerically and compared against other existing preconditioners.
Figure \ref{Vis} illustrates the CPU times of the $\mathcal{P}_{RD}$ preconditioned MINRES method for different viscosities ($\nu=1, 0.1, 0.01$) as the grid size increases.

The results clearly demonstrate that the CPU times are sensitive to the viscosity parameter $\nu$.
For $\nu=1$, the  $\mathcal{P}_{RD}$ preconditioner remains highly efficient, with CPU times staying relatively moderate even for the largest grid.
As $\nu$ decreases to $0.1$, the CPU times grow significantly, indicating that the preconditioner experiences some degradation in performance. For $\nu=0.01$, the computational cost further increases, reflecting the stronger ill-conditioning of the system matrix.

\section*{Declarations
}
\section*{Funding}
This work did not receive any financial support.
\section*{Conflicts of Interest Statement}
The author declares no conflicts of interest.
\section*{Author contribution}
Material preparation, data collection, and analysis were performed by Achraf Badahmane. The first draft of the manuscript was written by Achraf Badahmane. The final version of the manuscript was read and approved by Achraf Badahmane.

\bibliographystyle{elsarticle-num-names}
\bibliographystyle{plain}

\end{document}